\documentclass[12pt,leqno]{article}
\usepackage{graphics,graphicx,amssymb}
\usepackage{amsmath,amsthm}
\numberwithin{equation}{section}
\newtheorem{thm}{Theorem}[section]
\newtheorem{lem}[thm]{Lemma}
\newtheorem{prop}[thm]{Proposition}
\newtheorem{cor}[thm]{Corollary}

\theoremstyle{definition}

\theoremstyle{remark}

\newcommand{\X}{\mathfrak{X}}
\newcommand{\s}{\mathfrak{S}}
\newcommand{\W}{\mathcal{W}}
\newfont{\w}{msbm9 scaled\magstep1}

\newcommand{\norm}[1]{\left\Vert#1\right\Vert ^2}

\newcommand{\secref}[1]{\S\ref{#1}}

\begin{document}

\title
{On the geometry of quasi-K\"{a}hler manifolds with Norden metric}

\author{Dimitar Mekerov, Mancho Manev}

\maketitle
\begin{abstract}
The basic class of the non-integrable almost complex manifolds
with Norden metric is considered. Its curvature properties are
studied. The isotropic K\"ahler type of investigated manifolds is
introduced and characterized geometrically.


\thanks{\emph{Key words and phrase}:
almost complex manifold, Norden metric, quasi-K\"ahler manifold,
indefinite metric, non-integrable almost complex structure, Lie
group}
\end{abstract}


The generalized $B$-manifolds are introduced in \cite{GrMeDj}.
They are also known as almost complex manifolds with Norden metric
in \cite{GaBo} and as almost complex manifolds with B-metric in
\cite{GaGrMi}. In the present paper these manifolds are called
almost complex manifolds with Norden metric.

The aim of the present work is to further study of the geometry of
one of the basic classes of almost complex manifolds with Norden
metric. This is the class of the quasi-K\"ahler manifolds with
Norden metric, which is the only basic class with non-integrable
almost complex structure.

In \secref{sec_1} we recall the notions of the almost complex
manifolds with Norden metric, we give some of their curvature
properties and introduce isotropic K\"ahler type of the considered
manifolds.

In \secref{sec_2} we specialize some curvature properties for the
quasi-K\"ahler manifolds with Norden metric and the corresponding
invariants.


\section{Almost Complex Manifolds with Norden Metric}\label{sec_1}

Let $(M,J,g)$ be a $2n$-dimensional almost complex manifold with
Norden metric, i.e. $J$ is an almost complex structure and $g$ is
a metric on $M$ such that
\begin{equation}\label{Jg}
J^2X=-X, \qquad g(JX,JY)=-g(X,Y)
\end{equation}
for all differentiable vector fields $X$, $Y$ on $M$, i.e. $X, Y
\in \X(M)$.

The associated metric $\tilde{g}$ of $g$ on $M$ given by
$\tilde{g}(X,Y)=g(X,JY)$ for all $X, Y \in \X(M)$ is a Norden
metric, too. Both metrics are necessarily of signature $(n,n)$.
The manifold $(M,J,\tilde{g})$ is an almost complex manifold with
Norden metric, too.

Further, $X$, $Y$, $Z$, $U$ ($x$, $y$, $z$, $u$, respectively)
will stand for arbitrary differentiable vector fields on $M$
(vectors in $T_pM$, $p\in M$, respectively).

The Levi-Civita connection of $g$ is denoted by $\nabla$. The
tensor filed $F$ of type $(0,3)$ on $M$ is defined by
\begin{equation}\label{F}
F(X,Y,Z)=g\bigl( \left( \nabla_X J \right)Y,Z\bigr).
\end{equation}

It has the following symmetries
\begin{equation}\label{F-prop}
F(X,Y,Z)=F(X,Z,Y)=F(X,JY,JZ).
\end{equation}

Let $\{e_i\}$ ($i=1,2,\dots,2n$) be an arbitrary basis of $T_pM$
at a point $p$ of $M$. The components of the inverse matrix of $g$
are denoted by $g^{ij}$ with respect to the basis $\{e_i\}$.

The Lie form $\theta$ associated with $F$ is defined by
\begin{equation}\label{theta}
\theta(z)=g^{ij}F(e_i,e_j,z).
\end{equation}

A classification of the considered manifolds with respect to $F$
is given in \cite{GaBo}. Eight classes of almost complex manifolds
with Norden metric are characterized there by conditions for $F$.
The three basic classes are given as follows
\begin{equation}\label{class}
\begin{array}{l}
\W_1: F(x,y,z)=\frac{1}{4n} \left\{
g(x,y)\theta(z)+g(x,z)\theta(y)\right. \\[4pt]
\phantom{\mathcal{W}_1: F(x,y,z)=\frac{1}{4n} }\left.
    +g(x,J y)\theta(J z)
    +g(x,J z)\theta(J y)\right\};\\[4pt]
\W_2: \mathop{\s} \limits_{x,y,z}
F(x,y,J z)=0,\quad \theta=0;\\[8pt]
\W_3: \mathop{\s} \limits_{x,y,z} F(x,y,z)=0,
\end{array}
\end{equation}
where $\s $ is the cyclic sum by three arguments.

The special class $\W_0$ of the K\"ahler manifolds with Norden
metric belonging to any other class is determined by the condition
$F=0$.

Let $R$ be the curvature tensor field of $\nabla$ defined
by
\begin{equation}\label{R}
    R(X,Y)Z=\nabla_X \nabla_Y Z - \nabla_Y \nabla_X Z -
    \nabla_{[X,Y]}Z.
\end{equation}
The corresponding tensor field of type $(0,4)$ is
determined as follows
\begin{equation}\label{R04}
    R(X,Y,Z,U)=g(R(X,Y)Z,U).
\end{equation}

\begin{thm} \label{t1}
Let $(M,J,g)$ be an almost complex manifold with Norden metric.
Then the following identities are valid
\begin{enumerate}
\renewcommand{\labelenumi}{(\roman{enumi})}
    \item
    $R(X,Y,JZ,U)-R(X,Y,Z,JU)=
    \left(\nabla_X F\right)(Y,Z,U)-\left(\nabla_Y
    F\right)(X,Z,U); $
    \item
    $\left(\nabla_X F\right)(Y,JZ,U)+\left(\nabla_X
    F\right)(Y,Z,JU)\\
    =
    -g\bigl(\left(\nabla_X J\right)Z,\left(\nabla_Y J\right)U\bigr)
    -g\bigl(\left(\nabla_X J\right)U,\left(\nabla_Y
    J\right)Z\bigr);$
    \item
    $\left(\nabla_X F\right)(Y,Z,U)=\left(\nabla_X
    F\right)(Y,U,Z);$
    \item
    $g^{ij}\left(\nabla_{e_i} F\right)(e_j,Jz,u)+
    g^{ij}\left(\nabla_{e_i} F\right)(e_j,z,Ju)\\
    =-2g^{ij}g\bigl(\left(\nabla_{e_i} J\right)z,\left(\nabla_{e_j}
    J\right)u\bigr)$.
\end{enumerate}
\end{thm}
\begin{proof}
The equality (i) follows from the Ricci identity for $J$
\[
\bigl(\nabla_X\nabla_Y J\bigr)Z-\bigl(\nabla_Y\nabla_X
J\bigr)Z=R(X,Y)JZ - JR(X,Y)Z
\]
and the property of covariant constancy of $g$, i.e.
$\nabla g=0$.

The property \eqref{F-prop} of $F$ and the definition of
the covariant derivative of $F$ imply the equations (ii)
and (iii).

The equation (iv) is a corollary of (ii) by the action of
contraction of $X=e_i$ and $Y=e_j$ for an arbitrary basis
$\{e_i\} (i=1, 2, \dots, 2n)$ of $T_pM$.
\end{proof}

The square norm $\norm{\nabla J}$ of $\nabla J$ is defined
by
\begin{equation}\label{snorm}
    \norm{\nabla J}=g^{ij}g^{kl}
    g\bigl(\left(\nabla_{e_i} J\right)e_k,\left(\nabla_{e_j}
    J\right)e_l\bigr).
\end{equation}

A manifold $(M,J,g)$ belongs to the class $\W_0$ if and only if
$\nabla J=0$. It is clear that if $(M,J,g)\in\W_0$, then
$\norm{\nabla J}$ vanishes, too, but the inverse proposition is
not always true. That is, in general, the vanishing of the square
norm $\norm{\nabla J}$ does not always imply the K\"ahler
condition $\nabla J=0$.

An almost complex manifold with Norden metric satisfying the
condition $\norm{\nabla J}$ to be zero is called an
\emph{isotropic K\"ahler manifold with Norden metric}.

A special subclass of the investigated manifolds consisting of
isotropic K\"ahler but non-K\"ahler manifold with Norden metric is
considered in \cite{Me}. In the next section we will focus on this
case.


\section{The Quasi-K\"ahler Manifolds with Norden Metric}\label{sec_2}

Let $(M,J,g)$ be a quasi-K\"ahler manifold with Norden metric (in
short a $\W_3$-mani\-fold), i.e. it belongs to the class $\W_3$.
\begin{prop}
The following properties are valid for an arbitrary
$\W_3$-manifold.
\begin{enumerate}
\renewcommand{\labelenumi}{(\roman{enumi})}
    \item
    $\bigl(\nabla_X J\bigr)JY+\bigl(\nabla_Y J\bigr)JX+
    \bigl(\nabla_{JX} J\bigr)Y+\bigl(\nabla_{JY}
    J\bigr)X=0$;
    \item
    $\mathop{\s} \limits_{X,Y,Z}
    F(JX,Y,Z)=0$;
    \item
    $\mathop{\s} \limits_{Y,Z,U}
    \bigl(\nabla_X F\bigr)(Y,Z,U)=0$;
    \item
    $g^{ij}\bigl(\nabla_X F\bigr)(e_i,e_j,z)=g^{ij}\bigl(\nabla_X
    F\bigr)(z,e_i,e_j)=0$,
\end{enumerate}
where $\s$ is the cyclic sum by three arguments.
\end{prop}
\begin{proof}
The equalities (i) and (ii) are equivalent to the
characteristic condition \eqref{class} for the class
$\W_3$. The defining equation of the covariant derivative
of $F$, the condition $\nabla g=0$ and the definition
\eqref{class} of $\W_3$ imply the equalities (iii) and
(iv).
\end{proof}

In \cite{Me} it is proved that on every $\W_3$-manifold the
curvature tensor $R$ satisfies the following identity
\begin{equation}\label{12}
\begin{array}{l}
  R(X,JZ,Y,JU)+R(X,JY,U,JZ)+R(X,JY,Z,JU)
  \\[4pt]
  +R(X,JZ,U,JY)+R(X,JU,Y,JZ)+R(X,JU,Z,JY)
  \\[4pt]
  +R(JX,Z,JY,U)+R(JX,Y,JU,Z)+R(JX,Y,JZ,U)
  \\[4pt]
  +R(JX,Z,JU,Y)+R(JX,U,JY,Z)+R(JX,U,JZ,Y)
  \\[4pt]
  =-\mathop{\s}\limits_{X,Y,Z}
    g\Bigl(\bigl(\nabla_X J\bigr)Y+\bigl(\nabla_Y J\bigr)X,
     \bigl(\nabla_Z J\bigr)U+\bigl(\nabla_U J\bigr)Z\Bigr). \\[4pt]
\end{array}
\end{equation}

Let us consider an associated tensor of the Ricci tensor $\rho$
defined by the equation $\rho^*(y,z)=g^{ij}R(e_i,y,z,Je_j)$ on an
almost complex manifold with Norden metric. The tensor $\rho^*$ is
symmetric because of the first Bianchi identity.

By virtue of the identity \eqref{12} we get immediately the
following
\begin{lem}
For a $\W_3$-manifold $(M,J,g)$ with the Ricci tensor $\rho$ of
$\nabla$ and its associated tensor
$\rho^*(y,z)=g^{ij}R(e_i,y,z,Je_j)$ we have
\begin{enumerate}
\renewcommand{\labelenumi}{(\roman{enumi})}
    \item
    $
    \begin{array}{l}
      \rho^*(Jy,z)+\rho^*(y,Jz)+\rho(y,z)-\rho(Jy,Jz)
      \\[4pt]
      =-g^{ij}g\Bigl(\bigl(\nabla_{e_i} J\bigr)y+\bigl(\nabla_y J\bigr)e_i,
     \bigl(\nabla_z J\bigr)e_j+\bigl(\nabla_{e_j} J\bigr)z\Bigr);
     \\[4pt]
    \end{array}
    $
    \item
    $\norm{\nabla J}=-2g^{ij}g^{kl}
    g\Bigl(\bigl(\nabla_{e_i} J\bigr)e_k,\bigl(\nabla_{e_l} J\bigr)e_j\Bigr)$.
\end{enumerate}
\end{lem}

The last lemma and the identity \eqref{12} imply the next
\begin{thm}
Let $(M,J,g)$ be a $\W_3$-manifold. Then
\[
\norm{\nabla J}=-2(\tau+\tau^{**}),
\]
where $\tau$ is the scalar curvature of $\nabla$ and
$\tau^{**}=g^{ij}g^{kl}R(e_i,e_k,Je_l,Je_j)$.
\end{thm}
Hence, the last theorem implies the following
\begin{cor}
If $(M,J,g)$ is an isotropic K\"ahler $\W_3$-manifold then
$\tau^{**}=-\tau$.
\end{cor}

According to \cite{Me}, if $(M,J,g)$, $\dim{M}\geq 4$, is a
$\W_3$-manifold which has the K\"ahler property of $R$:
$R(X,Y,JZ,JU)=-R(X,Y,Z,U)$, then the norm $\norm{\bigl(\nabla_x
J\bigr)x}$ vanishes for every vector $x\in T_pM$.

Since $\norm{\bigl(\nabla_x J\bigr)x}= g\Bigl(\bigl(\nabla_{x}
J\bigr)x,\bigl(\nabla_{x} J\bigr)x\Bigr)=0$ holds, then applying
the substitutions $x\rightarrow x+y$, primary, and $x\rightarrow
x+z$, $y\rightarrow y+u$, secondary, we receive the following
condition
\[
g\Bigl(\bigl(\nabla_{x} J\bigr)z,\bigl(\nabla_{y} J\bigr)u\Bigr)
+g\Bigl(\bigl(\nabla_{x} J\bigr)u,\bigl(\nabla_{y} J\bigr)z\Bigr)
=0.
\]
Therefore, for the traces we have
\[
g^{ij}g^{kl}g\Bigl(\bigl(\nabla_{e_i}
J\bigr)e_k,\bigl(\nabla_{e_j} J\bigr)e_l\Bigr)=0.
\]
Having in mind \eqref{snorm} we obtain the following
\begin{prop}
If $(M,J,g)$, $\dim{M}\geq 4$, is a $\W_3$-manifold with K\"ahler
curvature tensor, then it is isotropic K\"ahlerian.
\end{prop}

Let $\alpha_1$ and $\alpha_2$ be holomorphic 2-planes determined
by the basis $(x,Jx)$ and $(y,Jy)$, respectively. The holomorphic
bisectional curvature $h(x,y)$ of the pair of holomorphic 2-planes
$\alpha_1$ and $\alpha_2$ is introduced in \cite{GrDjMe} by the
following way
\begin{equation}\label{h}
    h(x,y)=-\frac{R(x,Jx,y,Jy)}
    {\sqrt{\left\{g(x,x)\right\}^2+\left\{g(x,Jx)\right\}^2}
    \sqrt{\left\{g(y,y)\right\}^2+\left\{g(y,Jy)\right\}^2}},
\end{equation}
where $x$, $y$ do not lie along the totally isotropic directions,
i.e. the quadruple of real numbers $ \bigl(g(x,x), g(x,Jx),
g(y,y), g(y,Jy)\bigr)$ does not coincide with the null quadruple
$\left(0,0,0,0\right)$. The holomorphic bisectional curvature is
invariant with respect to the basis of the 2-planes $\alpha_1$ and
$\alpha_2$. In particular, if $\alpha_1=\alpha_2$, then the
holomorphic bisectional curvature coincides with the holomorphic
sectional curvature of the 2-plane $\alpha_1=\alpha_2$.

For arbitrary vectors $x,y \in T_pM$ we have the following
equation
\[
    R(x,Jx,y,Jy)=-h(x,y)
    \sqrt{\left\{g(x,x)\right\}^2+\left\{g(x,Jx)\right\}^2}
    \sqrt{\left\{g(y,y)\right\}^2+\left\{g(y,Jy)\right\}^2}.
\]
Let us note that $R(x,Jx,y,Jy)$ is the main component of the
curvature tensor on the 4-dimensional holomorphic space spanned by
the frame $\{x,y,Jx,Jy\}$, which is contained in $T_pM$.
Therefore, an important problem is the vanishing of
$R(x,Jx,y,Jy)$.

\begin{thm}\label{t21}
Let $(M,J,g)$, $\dim{M}\geq 4$, be a $\W_3$-manifold and $x,y \in
T_pM$. On a 4-dimensional holomorphic space spanned by
$\{x,y,Jx,Jy\}$ in $T_pM$ the condition $R(x,Jx,y,Jy)=0$ is
equivalent to the truthfulness of at least one of the following
conditions
\begin{enumerate}
\renewcommand{\labelenumi}{(\roman{enumi})}
    \item
    the bisectional curvature of holomorphic 2-planes $\{x,Jx\}$ and $\{y,Jy\}$ vanishes, i.e.
    $h(x,y)=0$;
    \item
    the holomorphic 2-plane $\{x,Jx\}$ is strongly isotropic, \\
    i.e.
    $g(x,x)=g(x,Jx)=0$;
    \item
    the holomorphic 2-plane $\{y,Jy\}$ is strongly isotropic, \\
    i.e.
    $g(y,y)=g(y,Jy)=0$.
\end{enumerate}
\end{thm}
\begin{thm}\label{t22}
Let $(M,J,g)$, $\dim{M}\geq 4$, be a $\W_3$-manifold and
$R(x,Jx,y,Jy)=0$ for all $x,y \in T_pM$. Then $(M,J,g)$ is an
isotropic K\"ahler $\W_3$-manifold.
\end{thm}
\begin{proof}
Let the condition $R(x,Jx,y,Jy)=0$ be valid. At first we
substitute $x+z$ and $y+u$ for $x$ and $y$, respectively.
According to the first Bianchi identity and \eqref{12}, we get
\begin{equation}
R(x,Jy,z,Ju)+R(Jx,y,Jz,u)-R(x,Jy,Jz,u)-R(Jx,y,z,Ju)=0
\end{equation}
and two similar equations which imply the vanishing of the left
side of \eqref{12}. Therefore
\[
    \mathop{\s}\limits_{x,y,z}
    g\Bigl(\bigl(\nabla_x J\bigr)y+\bigl(\nabla_y J\bigr)x,
    \bigl(\nabla_z J\bigr)u+\bigl(\nabla_u J\bigr)z\Bigr)=0.
\]
By contracting the last equation, having in mind \eqref{snorm}, we
receive the condition $\norm{\nabla J}=0$.
\end{proof}

\vskip2em \noindent
University of Plovdiv\\
Faculty of Mathematics and Informatics\\
236 Bulgaria blvd.\\
Plovdiv 4003, Bulgaria \\
mircho@pu.acad.bg, mmanev@yahoo.com

\end{document}